\documentclass[reqno,12pt]{amsart}
\usepackage{amsmath,amsfonts,amsthm,amssymb,indentfirst}
\usepackage[all,poly]{xy} 
\usepackage{color}
\usepackage[pagebackref,colorlinks]{hyperref}
\usepackage{kantlipsum}
\usepackage[all,poly]{xy}
\usepackage{mathrsfs}
\usepackage{enumerate}
\theoremstyle{plain}
\newtheorem{lemma}{Lemma}[section] 
\newtheorem{theorem}[lemma]{Theorem}
\newtheorem{corollary}[lemma]{Corollary}
\newtheorem{proposition}[lemma]{Proposition}

\theoremstyle{definition}
\newtheorem{remark}[lemma]{Remark}
\newtheorem{example}[lemma]{Example}
\newtheorem{definition}[lemma]{Definition}

\newcommand{\Zset}{\mathbb Z}
\newcommand{\Rset}{\mathbb R}

\newcommand{\V}{\mathbf{V}}
\newcommand{\N}{\mathbf{N}}
\newcommand{\ann}{\operatorname{ann}}
\newcommand{\red}{\operatorname{red}}

\newcommand{\so}{\mathbf{s}}
\newcommand{\ra}{\mathbf{r}}

\title[Graded irreducible representations: a new type and complete classification]{Graded irreducible representations of Leavitt path algebras: a new type and complete classification}  

\author{Lia Va\v s}
\address{Department of Mathematics, Saint Joseph's University, Philadelphia, PA 19131, USA}
\email{lvas@sju.edu}

\subjclass{16S88, 16G99, 16W50, 16D60} 
\keywords{Leavitt path algebra, irreducible representation, graded ideal, primitive ideal}  

\begin{document}
\begin{abstract} 
We present a new class of graded irreducible representations of a Leavitt path algebra. This class is new in the sense that its representation space is not isomorphic to any of the existing simple Chen modules.
The corresponding graded simple modules complete the list of Chen modules which are graded, creating an exhaustive class: the annihilator of any graded simple module is equal to the annihilator of either a graded Chen module or a module of this new type.  

Our characterization of graded primitive ideals of a Leavitt path algebra in terms of the properties of the underlying graph is the main tool for proving the completeness of such classification. We also point out a problem with the characterization of primitive ideals of a Leavitt path algebra in [K. M. Rangaswamy, Theory of prime ideals of Leavitt path algebras over arbitrary graphs, J. Algebra  375 (2013), 73 -- 90]. 
\end{abstract}

\maketitle

\section{Introduction}

In a series of papers, \cite{Chen}, \cite{Ara_Ranga}, and \cite{Ranga_simple_modules} (listed in chronological order), four classes of irreducible representations of Leavitt path algebras were introduced. The corresponding representation spaces have often been referred to as the Chen modules. The four types of Chen modules are exhaustive in the sense that the annihilator of any simple module over a Leavitt path algebra is equal to the annihilator of a Chen module. 

Leavitt path algebras are naturally graded by the group of integers. Because of the favorable properties of the graded structure (for example, \cite[Theorem 9]{Roozbeh_regular}, \cite[Theorem 5.7]{Tomforde}, and \cite[Theorem 3.3]{Lia_porcupine}), the study of {\em graded} irreducible representations is at least as relevant as the study of irreducible representations. 
One of the first questions that arises is whether the Chen modules which are graded represent a complete list of graded simple modules up to the equality of their annihilators. 

We show that this list is not complete and introduce the missing class of graded simple modules which completes the list. A module in this class is denoted by $\N_c$ and it is defined for a cycle $c$ such that no vertex on $c$ is on another distinct cycle. Such cycles have been called exclusive in \cite{Ara_Ranga} and we adopt this name (in \cite{Ranga_prime}, such cycles are said to be ``without $K$''). A module of the form $\N_c$ is graded simple and  it is not isomorphic to any of the Chen modules (Proposition \ref{proposition_N_c}). We use the term {\em graded Chen modules} to refer to Chen modules which are graded or modules of the form $\N_c.$ In Theorem \ref{graded_simple}, we show that the annihilator of any graded simple module over a Leavitt path algebra is equal to the annihilator of a graded Chen module.  

The proof of Theorem \ref{graded_simple} uses Theorem \ref{graded_primitive_ideal} which characterizes graded primitive ideals of a Leavitt path algebra in terms of the properties of the corresponding admissible pairs (an admissible pair consists of two sets of vertices of the graph which naturally correspond to a graded ideal). 

In Corollary \ref{corollary_primitive_ideal}, we characterize primitive ideals of a Leavitt path algebra and, in section \ref{subsection_Ranga}, we point out an issue in a part of Theorem 4.3 of \cite{Ranga_prime}. Example \ref{example_CSP_not_ISCP} has more details on this matter. 

In section \ref{subsection_RangaRoozbeh_N_vc}, we discuss the definition of the module $\N_{vc}$ from \cite{Roozbeh_Ranga} where $c$ is a cycle without exits and $v$ is a vertex. In section \ref{subsection_RangaRoozbeh_annihilator}, we correct  the formula for the annihilator of $\N_{vc}$ in \cite[Theorem 3.5]{Roozbeh_Ranga}. 

We also point out that parts of this work and parts of \cite{Loc_Van} complement each other. In particular,
our module $\N_c$ can be incorporated in \cite[Proposition 4.21]{Loc_Van} and the proof of \cite[Proposition 4.22]{Loc_Van} can be extended to also include exclusive cycles with exits, not only cycles without exits.

\section{Prerequisites}\label{section_prerequisites}

All rings are associative but not necessarily unital. Modules are left modules unless noted otherwise. A reference to an ideal without left or right specifications indicates that the ideal is two-sided. For a module $M$ over a ring $R,$ we let $\ann_R(M)$ (or $\ann(M)$ if it is clear from the context what $R$ is) denote the annihilator of $M,$ i.e. the set $\{r\in R\mid rm=0$ for all $m\in M\}.$ We recall that $\ann_R(M)$ is an ideal.

\subsection{Graded rings} \label{subsection_graded_rings}
A ring $R$ is {\em graded} by a group $\Gamma$ if $R=\bigoplus_{\gamma\in\Gamma} R_\gamma$ for additive subgroups $R_\gamma$ and $R_\gamma R_\delta\subseteq R_{\gamma\delta}$ for all $\gamma,\delta\in\Gamma.$ The elements of the set $\bigcup_{\gamma\in\Gamma} R_\gamma$ are said to be {\em homogeneous}. The grading is {\em trivial} if $R_\gamma=0$ for every $\gamma\in \Gamma$ which is not the group identity. An $R$-module $M$ is {\em graded} if $M=\bigoplus_{\gamma\in\Gamma} M_\gamma$ for additive subgroups $M_\gamma$ and $R_\gamma M_\delta\subseteq M_{\gamma\delta}$ for all $\gamma,\delta\in\Gamma.$ If $M$ is a graded left $R$-module and $\gamma\in\Gamma,$ the $\gamma$-\emph{shifted} graded module $M(\gamma)$ is defined as the module $M$ with the $\Gamma$-grading given by $M(\gamma)_\delta = M_{\delta\gamma}$ for all $\delta\in \Gamma.$
A left ideal $I$ of a graded ring $R$ is {\em graded} if for any $x\in I$ and any $\gamma\in \Gamma,$ the $\gamma$-component $x_\gamma$ of $x$ is in $I$. Graded right ideals and graded ideals are defined similarly. We adopt the standard definitions of graded ring homomorphisms, graded module homomorphisms, graded fields, and graded algebras as defined in \cite{Roozbeh_book}. 

If $P$ is a property of ring modules, we say that a graded module $M$ has {\em graded property $P$} if a statement analogous to $P$ but considered in the category of graded modules, not the category of modules, holds for $M.$ For example, a module is {\em graded simple} if it does not have any proper and nontrivial {\em graded} submodules. As another example, a graded ideal $I$ is {\em graded prime} if $JK\subseteq I$ implies that $J\subseteq I$ or $K\subseteq I$ for any two {\em graded} ideals $J$ and $K.$ Following the proof of an analogous claim in the non-graded case, it is direct to show that a graded ideal $I$ is graded prime if and only if $aRb\subseteq I$ implies $a\in I$ or $b\in I$ for every two {\em homogeneous} elements $a,b\in R.$

If $P$ is a ring property defined using a module property $Q,$ then the {\em graded property $P$} is obtained from $P$ considering  graded $Q$ instead of $Q.$ For example, a graded ring is said to be {\em graded prime} if $\{0\}$ is a {\em graded} prime ideal. 

If an ideal $I$ of a graded ring $R$ is graded and prime, it is a graded prime ideal. The converse does not hold in general, but it holds if $\Gamma$ is the set of integers $\Zset$ as the following lemma states. 
\begin{lemma} \cite[Proposition II.1.4]{NvO_book}
An ideal $I$ of a $\Zset$-graded ring $R$ is graded prime if and only if it is graded and prime. 
\label{Zset_graded_prime}
\end{lemma} 

\subsection{Graded primitive rings}\label{subsection_graded_primitrive_rings}
A $\Gamma$-graded ring is said to be {\em graded left primitive} if there is a graded simple left module $M$ such that $\ann(M)=0.$ A graded ideal $I$ of $R$ is {\em graded left primitive} if $R/I$ is a graded left primitive ring (equivalently, if there is a graded simple left $R$-module with $I$ as its annihilator). Graded right primitive rings and graded right primitive ideals are similarly defined and a graded ring or a graded ideal is graded primitive if it is both graded left and graded right primitive. 

The implication graded left primitive $\Rightarrow$ graded prime holds for graded ideals over locally unital rings  and it can be shown by ``grading'' the well-known argument for left primitive $\Rightarrow$ prime. Theorem \ref{graded_primitive_ideal} and Corollary \ref{corollary_primitive_ideal} enable us to construct quick examples (see the example in the last paragraph of section \ref{section_graded_primitive}) of graded primitive ideals which are not primitive.

We note the following three statements obtained as graded versions of well-known (non-graded) results. The proofs of these statements can be obtained by directly following the proofs of their non-graded versions in papers listed together with the statements themselves.  

\begin{lemma} (The graded version of \cite[Lemmas 1 and 2]{Lanski_Resco_Small})
If $R$ is a graded prime algebra over a graded field $K,$ then there is a graded prime unital $K$-algebra $R^u$ into which $R$ embeds as a graded ideal and such that $R^u$ is graded primitive if and only if $R$ is graded primitive.
\label{lemma_on_unitization}
\end{lemma}

\begin{lemma} (The graded version of \cite[Lemma 11.28]{Lam_first_course}) A $\Gamma$-graded unital ring $R$ is graded left primitive if and only if there is a proper graded left ideal $I_0$ such that for every nontrivial graded (two-sided) ideal $I$, $I_0+I=R.$
\label{lemma_graded_primitive_characterization}
\end{lemma}

\begin{lemma} (The graded version of \cite[Theorem 1.1]{Fisher_Snider}) Let $R$ be a graded prime ring in which each nontrivial graded left ideal contains a nonzero homogeneous idempotent. If there are nontrivial graded ideals $I_n$ where $n$ is a positive integer such that for every nontrivial graded ideal $I,$ there is $n$ such that $I\cap I_n\neq 0,$ then $R$ is graded primitive. 
\label{lemma_Fisher_Snider}
\end{lemma}

\subsection{Graphs and properties of vertex sets}
\label{subsection_graphs} 

Let $E$ be a directed graph, let $E^0$ denote the set of vertices, $E^1$ denote the set of edges, and $\so$ and $\ra$ denote the source and the range maps of $E.$ A {\em sink} of $E$ is a vertex which does not emit edges and an {\em infinite emitter} is a vertex which emits infinitely many edges. A vertex of $E$ is {\em regular} if it is not a sink nor an infinite emitter. 

A {\em path} is a sequence of edges $e_1e_2\ldots e_n$ such that  $\ra(e_i)=\so(e_{i+1})$ for $i=1,\ldots, n-1$ or a single vertex. The set of vertices on a path $p$ is denoted by $p^0.$  The functions $\so$ and $\ra$ extend to paths naturally. A path $p$ is {\em closed}  if $\so(p)=\ra(p).$ A {\em cycle} is a closed path such that different edges in the path have different sources. A cycle has {\em an exit} if a vertex on the cycle emits an edge outside of the cycle. A subset $V$ of $E^0$ has {\em Condition (L)} if every cycle with vertices in $V$ has an exit whose range is in $V$ also. If $V=E^0,$ we say that $E$ has Condition (L). A cycle is {\em exclusive} if ``no exit returns'' or, more formally, no vertex on $c$ is the base of a cycle distinct from $c.$ On the other side of the spectrum are the {\em extreme} cycles: cycles such that ``every exit returns'' i.e., cycles $c$ with exits and such that for every path $p$ with $\so(p)\in c^0,$ there is a path $q$ such that $\ra(p)=\so(q)$ and $\ra(q)\in c^0.$ If $V\subseteq E^0,$ a cycle $c$ is 
{\em extreme in $V$} if $c^0\subseteq V,$ $c$ has an exit with range in $V$, and for every path $p$ with $\so(p)\in c^0$ and $p^0\subseteq V,$ there is a path $q$ such that $q^0\subseteq V,$ $\ra(p)=\so(q),$ and $\ra(q)\in c^0.$

An {\em infinite path} is a sequence of edges $e_1e_2\ldots$ such that $\ra(e_n)=\so(e_{n+1})$ for $n=1,2\ldots.$ If $\so(e_1)=v,$ we identify $\alpha$ and $v\alpha.$ Just as for finite paths, we use $\alpha^0$ for the set of vertices of an infinite path $\alpha.$ 
Two infinite paths $\alpha$ and $\beta$ are {\em tail equivalent} if there are positive integers $m$ and $n$ such that $\alpha$ without the prefix of length $m$ is equal to $\beta$ without the prefix of length $n.$
An infinite path is said to be {\em rational} if it is tail equivalent to a path of the form $cccc\ldots$ where $c$ is a cycle and it is said to be {\em irrational} otherwise. 

If $u,v\in E^0$ are such that there is a path $p$ with $\so(p)=u$ and $\ra(p)=v$, we write $u\geq v.$ Thus, if $p$ is a path $e_1e_2\ldots e_n,$ then the vertices of $p$ are not increasing in the sense that $\so(e_1)\geq \so(e_2)\geq\ldots \geq \so(e_n)\geq \ra(e_n)$ and an analogous statement holds for an infinite path. We say that an infinite path $\alpha$ is {\em strictly decreasing} if there are paths $p_n$ such that $\alpha=p_1p_2\ldots$ and such that $\so(p_1)\gneq \so(p_2)\gneq\ldots.$ A strictly decreasing infinite path $\alpha$ is irrational and $\alpha^0$ is infinite.   

Let $R(v)$ denote the set $\{u\in E^0\mid u\geq v\}$ for $v\in E^0$ and $R(V)$ denote the set  $\{u\in E^0\mid u\geq v$ for some $v\in V\}$ for $V\subseteq E^0.$ We call the set $R(V)$ {\em the root} of $V.$ We chose this name since the ``dual'' set $T(V)=\{u\in E^0\mid u\leq v$ for some $v\in V\}$ is called {\em the tree} of $V$. In \cite{Ranga_prime}, $R(V)$ is denoted by $M(V).$

A subset $H$ of $E^0$ is {\em hereditary} if $u\in H$ and $u\in R(v)$ implies $v\in H.$ The set $H$ is {\em saturated} if $v\in H$ for any regular vertex $v$ such that $\ra(\so^{-1}(v))\subseteq H.$ It is direct to check that $E^0-R(V)$ is hereditary and saturated for any $V\subseteq E^0.$   

A set $V$ of vertices is said to be {\em downwards directed} if for every two vertices $u,v\in V,$ there is $w\in V$ such that $u\geq w$ and $v\geq w.$ In \cite{Gene_Bell_Ranga}, the following concept was introduced: a set $V$ of vertices has {\em the Countable Separation Property} (CSP for short) if there is a countable set $C\subseteq E^0$ such that $V\subseteq R(C).$ We introduce a more restrictive notion. 
\begin{definition}
A set $V$ of vertices has the {\em Inner Countable Separation Property } (ICSP) if there is a countable set $C\subseteq V$ such that $V\subseteq R(C).$
\end{definition}
If $V=E^0,$ CSP and ICSP are equivalent. As Example \ref{example_CSP_not_ISCP} shows,  CSP $\nRightarrow$ ICSP.

\subsection{Leavitt path algebras, their graded ideals and quotients}
\label{subsection_LPAs} 
If $K$ is any field, the {\em path algebra} $P_K(E)$ of $E$ over $K$ is a free $K$-algebra generated by the set $E^0\cup E^1$ such that

\begin{tabular}{l}
(V) $vw =0$ if $v\neq w$ and $vv=v$ for $v,w\in E^0$\\
(E) $\so(e)e=e\ra(e)=e$ for $e\in E^1.$
\end{tabular}

The {\em extended graph} $\widehat{E}$ of a graph $E$ is the graph with vertices $E^0,$ edges $E^1\cup \{e^*\mid e\in E^1\},$ and with the range and the source maps the same on $E^1$ and given by $\so(e^*)=\ra(e)$ and $\ra(e^*)=\so(e)$ on the added edges. The \emph{Leavitt path algebra} $L_K(E)$ of $E$ over $K$ is the path algebra over the extended graph subject to the relations 

\begin{tabular}{l}
(CK1) $e^\ast f=0$ if $e\neq f$ and $e^\ast e=\ra(e)$ for edges $e,f$ and\\
(CK2) $v=\sum_{e\in \so^{-1}(v)} ee^\ast$ for each regular vertex $v.$\\
\end{tabular}

The axioms ensure that every element of $L_K(E)$ can be represented as a sum of the form $\sum_{i=1}^n k_ip_iq_i^\ast$ for some $n$, paths $p_i$ and $q_i$, and elements $k_i\in K,$ for $i=1,\ldots,n$  where $v^*=v$ for $v\in E^0$ and $p^*=e_n^*\ldots e_1^*$ for a path $p=e_1\ldots e_n.$ Using this representation, it is direct to see that $L_K(E)$ is locally unital (with the finite sums of vertices as the local units) and that $L_K(E)$ is unital if and only if $E^0$ is finite in which case the sum of all vertices is the identity. For more details on these basic properties, see \cite{LPA_book}.

If we consider $K$ to be trivially graded by $\Zset,$ $L_K(E)$ is naturally graded by $\Zset$ so that the $n$-component $L_K(E)_n$ is the $K$-linear span of the elements $pq^\ast$ for paths $p, q$ with $|p|-|q|=n$ where $|p|$ denotes the length of a path $p.$ While one can grade a Leavitt path algebra by any group (see \cite[Section 1.6.1]{Roozbeh_book}), we always consider the natural grading by $\Zset.$ 

If $H\subseteq E^0$ is hereditary and saturated, let 
$B_H$ be the set of infinite emitters $v$ in $E^0-H$ such that $\so^{-1}(v)\cap \ra^{-1}(E^0-H)$  is nonempty and finite. For such $v\in B_H,$ we let 
$v^H=v-\sum ee^*$ where the sum is taken over $e\in \so^{-1}(v)\cap \ra^{-1}(E^0-H).$

An {\em admissible pair} is a pair $(H, S)$ where $H\subseteq E^0$ is hereditary and saturated and $S\subseteq B_H.$ For such a pair, let $I(H,S)$ denote the graded ideal generated by homogeneous elements $H\cup \{v^H \mid  v\in S \}.$ 
The ideal $I(H,S)$ is the $K$-linear span of the elements $pq^*$ for paths $p,q$ with $\ra(p)=\ra(q)\in H$ and the elements $pv^Hq^*$ for paths $p,q$ with $\ra(p)=\ra(q)=v\in S$ (see \cite[Lemma 5.6]{Tomforde}). The converse holds as well: for every graded ideal $I$, the vertices in $I$ form a hereditary and saturated set $H$ and the set of infinite emitters such that $v^H\in I$ is a subset of $B_H$ (\cite[Theorem 5.7]{Tomforde}, also \cite[Theorem 2.5.8]{LPA_book}). Each admissible pair $(H,S)$ gives rise to the quotient graph $E/(H,S),$ defined as below.
\[
\begin{array}{l}
(E/(H,S))^0=E^0-H \cup\{v'\mid v\in B_H-S\} \\
(E/(H,S))^1=\{e\in E^1\mid \ra(e)\notin H\}\cup\{e'\mid e\in E^1\mbox{ and }\ra(e)\in B_H-S\} 
\end{array}
\]
with $\so$ and $\ra$ the same as on $E^1$ for $e\in E^1 \cap (E/(H,S))^1$ and $\so(e')=\so(e),$ $\ra(e')=\ra(e)'$ for $e'\in (E/(H,S))^1.$ The  quotient algebra $L_K(E)/I(H,S)$ is graded isomorphic to $L_K(E/(H,S))$ (see \cite[Theorem 5.7]{Tomforde}).

\subsection{Branching systems and classes of simple modules of Leavitt path algebras}\label{subsection_list_of_Chen_modules}

In \cite{Goncalves_Royer}, the study of representations of Leavitt path algebras using branching systems has been initiated. A {\em branching system} of a graph $E$ is a set $X$ with its subsets $X_v$ for $v\in E^0$ and $X_e$ for $e\in E^1,$ and maps $\sigma_e:  X_{\ra(e)}\to X_e$  for $e\in E^1$ such that the following holds. 
\begin{enumerate}
\item The sets $X_v, v\in E^0,$ are mutually disjoint and the sets $X_e, e\in E^1,$ are mutually disjoint. 
\item $X_e\subseteq X_{\so(e)}$ for every $e\in E^1.$ 
\item The map $\sigma_e$ is a bijection for every $e\in E^1.$  
\item $X_v=\bigcup_{e\in \so^{-1}(v)} X_{e}$ for each regular vertex $v.$ 
\end{enumerate}
If (4) holds also for infinite emitters, $X$ is {\em perfect}. If $X=\bigcup_{v\in E^0} X_v,$ $X$ is {\em saturated}. 

For a branching system $X$, one defines a $L_K(E)$-module $M(X)$ as follows: $M(X)$ is the vector space over $K$ with basis $X$ and with the $L_K(E)$-module structure induced by  
\begin{center}
\begin{tabular}{lllll}

$v\cdot x=x$ & if $x\in X_v$ & and & $v\cdot x=0$ & otherwise, \\
$e\cdot x=\sigma_e(x)$ & if $x\in X_{\ra(e)}$ & and &$e\cdot x=0$ & otherwise, and \\
$e^*\cdot x=\sigma_e^{-1}(x)$&  if $x\in X_e$& and & $e^*\cdot x=0$ & otherwise
\end{tabular}
\end{center}
for $v\in E^0, e\in E^1,$ and $x\in X.$   

In \cite[Definition 3.2]{Roozbeh_Ranga}, a branching system is said to be {\em graded} if there is a map $\deg: X\to \Zset$ such that $\deg(\sigma_e(x))=\deg(x)+1.$ By \cite[Section 3]{Roozbeh_Ranga}, if $X$ is graded, then $M(X)$ is graded where $M(X)_n$ for $n\in \Zset$ is the $K$-linear span of elements $x\in X$ with $\deg(x)=n.$ 

In \cite{Chen}, the first three of the following classes of irreducible representations or, equivalently, simple modules listed below were introduced. In \cite{Ara_Ranga} and \cite{Ranga_simple_modules}, some of these classes were further refined and class (4) was introduced. 
\begin{enumerate}
\item The infinite-path type. If $\alpha$ is an infinite path, let $X$ be the set of all infinite paths tail equivalent to $\alpha.$ If we let $X_v$ be $\{\beta\in X\mid \so(\beta)=v\},$ $X_e$ be $\{\beta\in X\mid e$ is the first edge of $\beta\},$ and $\sigma_e: X_{\ra(e)}\to X_e$ be given by $\beta\mapsto e\beta,$ then $X$ is a perfect and saturated branching system. The module $\V_{[\alpha]}=M(X)$ is simple. By \cite[Proposition 3.6]{Roozbeh_Ranga}, $\V_{[\alpha]}$ is graded if and only if $\alpha$ is irrational (so it is graded simple in this case). By \cite[Lemma 3.2]{Ara_Ranga}, if $\alpha$ is either irrational or rational but not tail equivalent to an exclusive cycle, then $\ann(\V_{[\alpha]})=I(H, B_H)$ for $H=E^0-R(\alpha^0)$ and if $\alpha$ is tail equivalent to $ccc\ldots$ for some 
exclusive cycle $c,$ then $\ann(\V_{[\alpha]})=I(H, B_H)+\langle c-v\rangle$ for $H=E^0-R(\alpha^0)$ and $v\in c^0.$
  
\item The twisted-$\V$ type. This type consists of modules denoted by $\V^f_{[\alpha]}$ where $\alpha$ is an infinite rational path and $f$ an irreducible polynomial over $K[x,x^{-1}].$ Since such a module is non-graded, it is not in our focus. We refer a reader to \cite[Section 3]{Ara_Ranga} for more details and exact definition. By \cite{Ara_Ranga}, $\V^f_{[\alpha]}$ is simple and $\ann(\V^f_{[\alpha]})=I(H, B_H)+\langle f(c)\rangle$ for $H=E^0-R(\alpha^0).$

\item The sink type. Let $v$ be a sink and $X$ be the set of all paths in the path algebra which end at $v$. If we let $X_w$ be $\{p\in X\mid \so(p)=w\},$ $X_e$ be $\{p\in X\mid p$ has $e$ as its first edge$\},$ and $\sigma_e: X_{\ra(e)}\to X_e$ be given by $p \mapsto ep,$ then $X$ is a perfect and saturated branching system. The module $\N_v=M(X)$ is simple, graded (thus graded simple) and, by \cite[Lemma 3.1]{Ara_Ranga}, $\ann(\N_v)=I(H, B_H)$ for $H=E^0-R(v).$

\item The infinite-emitter types. Let $v$ be an infinite emitter and $X,$ $X_w, X_e$ and $\sigma_e$ be defined as for the previous type. Then $X$ is a saturated branching system which is not perfect since $v\in X_v$ but $v\notin X_e$ for any $e\in E^1.$ Indeed, even in the case when $e\in \so^{-1}(v)\cap \ra^{-1}(R(v)),$ the elements $v$ and $ep$ are different basis elements of $M(X)$ for any $p\in X$ such that $\ra(e)=\so(p).$ The module $\N_v=M(X)$ is simple and graded.
To distinguish the cases when $\ra(\so^{-1}(v))\cap R(v)$  is empty, finite and nonempty, and infinite, we use different notations for $\N_v$ in these different cases. 
\begin{enumerate}
\item[] $\N_{v\emptyset}$ denotes $\N_v$ in the case when   $\ra(\so^{-1}(v))\cap R(v)$ is empty. 

\item[] $\N_{v\in B_H}$ denotes $\N_v$ in the case when $\ra(\so^{-1}(v))\cap R(v)$ is nonempty and finite. 

\item[] $\N_{v\infty}$  denotes $\N_v$ in the case when  $\ra(\so^{-1}(v))\cap R(v)$ is infinite. 
\end{enumerate} 
The annihilators of all three types are described in \cite[Proposition 2.4]{Ranga_simple_modules}: in the first and the third case, the annihilator is $I(H, B_H).$ In the second case, it is  $I(H, B_H-\{v\}).$
\end{enumerate}

We say that a module is a {\em Chen module} if it is one of the modules of types (1) to (4). 

While Chen modules of different types are not isomorphic, they can have the same annihilators. 

\begin{example}
Let $E$ be the graph consisting of a single vertex $v$ and infinitely many edges $e_n.$ For the infinite path $\alpha=e_1e_2\ldots,$ we have that $\ann(\V_{[\alpha]})=0=I(\emptyset, \emptyset)=\ann(\N_{v\infty}).$ The modules $\V_{[\alpha]}$ and $\N_{v\infty}$ are not isomorphic by \cite[Proposition 2.8]{Ranga_simple_modules}.  
\label{example_more_than_one_Chen} 
\end{example}

\section{Characterization of graded primitive ideals}\label{section_graded_primitive}

In the rest of the paper, $K$ is a field, $E$ is a graph, and $L_K(E)$ is the corresponding Leavitt path algebra. 

By \cite[Proposition 4.6]{Gene_Bell_Ranga}, if $L_K(E)$ is primitive, then $E^0$ has CSP. In Proposition \ref{proposition_sufficient}, we prove that the same statement holds if ``primitive'' is replaced by ``graded primitive''. This also follows from recently shown \cite[Theorem 2.11]{Ranga_graded_primitive}. As the argument in our proof is different, we include the proof also.  

The proof of \cite[Proposition 4.6]{Gene_Bell_Ranga} uses \cite[Proposition 4.4 and Lemma 4.5]{Gene_Bell_Ranga}.  Proposition 4.4 from \cite{Gene_Bell_Ranga} states that if $S\subseteq E^0$ does not have CSP, then there is $v\in S$ such that the set $Z\subseteq S,$ defined as in the proof below, does not have CSP. Since this statement makes no reference to either primitivity or graded primitivity, we can use it directly. Lemma 4.5 from \cite{Gene_Bell_Ranga} is about a certain ideal generated by a vertex. As such ideal is graded, we can use this lemma directly also.

\begin{proposition}
If $L_K(E)$ is graded primitive, then $E^0$ has CSP.  
\label{proposition_sufficient}
\end{proposition}
\begin{proof}
If $L_K(E)$ is graded primitive, then it is graded prime. 
By Lemma \ref{lemma_on_unitization}, there is a graded primitive unital algebra $A$ for which $L_K(E)$ is a graded ideal. By Lemma \ref{lemma_graded_primitive_characterization}, there is a proper graded left ideal $I$ such that $I+J=A$ for every nontrivial graded ideal $J$. Thus, for any vertex $v\in E^0,$ there is an element $x_v\in AvA$ such that $1+x_v\in I.$

By \cite[Lemma 4.5]{Gene_Bell_Ranga} and the proof of \cite[Proposition 4.6]{Gene_Bell_Ranga}, there is a sequence of subsets $S_n$ of $E^0$ with $\bigcup_n S_n=E^0.$ Assuming that $E^0$ does not have CSP, there is $n$ such that $S_n$ does not have CSP and, by \cite[Proposition 4.4]{Gene_Bell_Ranga}, a vertex $v\in S_n$ such that the set $Z=\{w \in S_n\mid x_v x_w = 0\}$ also does not have CSP. In particular, $Z\neq \emptyset,$ so let $w\in Z.$ Since $1+x_v, 1+x_w\in I,$ we have that 
$x_v=x_v+0=x_v+x_vx_w=x_v(1+x_w)\in I.$ Hence, $1=1+x_v-x_v\in I$ which a contradiction with the assumption that $I$ is proper. 
\end{proof}

This proposition enables us to prove direction $\Rightarrow$ in Theorem \ref{graded_primitive_LPA}. For the converse, we show Lemma \ref{lemma_nonzero_homogeneous_idempotent} which enables us to use Lemma \ref{lemma_Fisher_Snider}. Lemma \ref{lemma_nonzero_homogeneous_idempotent} is known to hold for a graded (two-sided) ideal for which the nonzero homogeneous idempotent from the statement of Lemma \ref{lemma_nonzero_homogeneous_idempotent} can be taken to be a vertex (see \cite[Lemma 4.1]{Tomforde}). 

\begin{lemma}
If $I$ is a nontrivial graded left ideal of $L_K(E),$ then $I$ contains a nonzero homogeneous idempotent.
\label{lemma_nonzero_homogeneous_idempotent}
\end{lemma}
\begin{proof}
If $I$ is a nontrivial graded left ideal of $L_K(E),$ then there is $a\in I$ such that $a\neq 0$ and $a$ is homogeneous. By \cite[Proposition 3.1]{Malaga_socle}, there are paths $p,q, r,s$ such that $pq^*ars^*$ is nonzero and either in $Kv$ for some $v\in E^0$ or in $vL_K(E)v$ for some $v$ which is on a cycle without exits. 

In the first case, we have that $pq^*ars^*=kv\neq 0$ for some $0\neq k\in K$ which implies that $\so(p)=\so(s)=v.$ Let $\varepsilon=k^{-1}rs^*pq^*a\in I$ which is homogeneous since $a$ is homogeneous. We have that $\varepsilon^2=k^{-2}rs^*pq^*ars^*pq^*a=k^{-2}rs^*\,kv\, pq^*a=k^{-1}rs^* pq^*a=\varepsilon.$ Since $\varepsilon rs^*=k^{-1}rs^*pq^*ars^*=k^{-1}rs^*kv=rs^*$ and $pq^*ars^*=kv\neq 0,$ we have that $\varepsilon rs^*\neq 0$ and this implies that $\varepsilon\neq0.$  

In the second case, let $c$ be a cycle without exits which contains $v$ and consider $c$ so that $\so(c)=\ra(c)=v.$ Let $c^{-1}$ denote $c^*$ and $c^0$ denote $v.$ If $m=|c|,$ a homogeneous element of $vL_K(E)v$ has the form $kc^{nm}$ where $k\in K$ and $n\in \Zset.$ Thus, $pq^*ars^*=kc^{nm}\neq 0$ for some $n$ and some $0\neq k\in K.$ This implies that $\so(p)=\so(s)=v.$ 
Let $\varepsilon=k^{-1}rs^*c^{-nm}pq^*a\in I$ which is homogeneous since $a$ is homogeneous. We have that $\varepsilon^2=k^{-2}rs^*c^{-nm}pq^*a rs^*c^{-nm}pq^*a=k^{-2}rs^*c^{-nm} kc^{nm} c^{-nm}pq^*a=k^{-1}rs^*c^{-nm} pq^*a=\varepsilon.$ Since $\varepsilon rs^*=k^{-1}rs^*c^{-nm}pq^*ars^*=k^{-1}rs^*c^{-nm}kc^{nm}=rs^*$  and $pq^*ars^*=kc^{nm}\neq 0,$ we have that $\varepsilon rs^*\neq 0$ and this implies that  $\varepsilon\neq0.$ 
\end{proof}

The content of Theorem \ref{graded_primitive_LPA} is also in \cite[Theorem 2.11]{Ranga_graded_primitive}. Since the proofs use different arguments, we include the proof of Theorem \ref{graded_primitive_LPA}. 

\begin{theorem}
$L_K(E)$ is graded primitive if and only if $E^0$ is downwards directed and has CSP. 
\label{graded_primitive_LPA}
\end{theorem}
\begin{proof}
If $L_K(E)$ is graded primitive, then 
$L_K(E)$ is graded prime, so it is prime by Lemma \ref{Zset_graded_prime}. Thus, $E^0$ is downwards directed by \cite[Theorem 1.4]{Gene_Bell_Ranga} and $E^0$ has CSP by Proposition \ref{proposition_sufficient}. 

To prove the converse, note that if $E$ is downwards directed, then $L_K(E)$ is a prime ring by \cite[Theorem 1.4]{Gene_Bell_Ranga}. Since $L_K(E)$ is graded, it is a graded prime ring. If $E^0$ has CSP, let $C=\{v_n\in E^0\mid n=1,2,\ldots\}$ be such that $E^0\subseteq R(C)$ and let $I_n$ be the ideal generated by $v_n$. Since it is generated by a homogeneous element, $I_n$ is graded. If $I$ is a nontrivial graded ideal, it contains a vertex (by \cite[Lemma 4.1]{Tomforde}). As $E^0\subseteq R(C),$ that vertex is in $R(v_n)$ for some $n,$ so $v_n\in I_n\cap I.$ By Lemma \ref{lemma_nonzero_homogeneous_idempotent}, the assumption of Lemma \ref{lemma_Fisher_Snider} is satisfied, so  $L_K(E)$ is graded primitive by Lemma \ref{lemma_Fisher_Snider}. 
\end{proof}

Using Theorem \ref{graded_primitive_LPA}, Lemma \ref{Zset_graded_prime}, and the characterization of prime Leavitt path algebras from \cite[Theorem 1.4]{Gene_Bell_Ranga}, it is direct to construct an example of a Leavitt path algebra which is graded prime and not graded primitive. For example, graph $F$ from Example \ref{example_CSP_not_ISCP} is such that $F^0$ is downwards directed and without CSP, so $L_K(F)$ is graded prime and not graded primitive. 

We turn to the characterization of graded primitive ideals now. We need the following four lemmas for the proof of Theorem \ref{graded_primitive_ideal}. 

\begin{lemma}
If $H\subseteq E^0$ is hereditary and if $E^0-H$ has ICSP, then $E^0-H=R(C)$ for a countable set $C\subseteq E^0-H.$  
\label{lemma_ISCP_and_hereditary}
\end{lemma}
\begin{proof}
Let $C$ be a countable set such that $C\subseteq E^0-H\subseteq R(C).$ This implies that $R(C)\subseteq R(E^0-H)\subseteq R(R(C))=R(C)$ and so $R(C)=R(E^0-H).$ The assumption that $H$ is hereditary implies that $R(E^0-H)=E^0-H.$ Indeed, assuming that $v\in R(E^0-H),$ there is $u\in E^0-H$ such that $v\geq u.$ Then $v\notin H$ since otherwise $u$ would be in $H$ as $H$ is hereditary. Thus, $v\in E^0-H.$ Hence, $R(C)=R(E^0-H)=E^0-H.$ 
\end{proof}

\begin{lemma}
If $H\subseteq E^0$ is hereditary and if $E^0-H$ is downwards directed with ICSP, then $E^0-H=R(C)$ for a countable set $C=\{v_n\in E^0-H\mid v_n\geq v_{n+1}$ and $v_{n+1}\ngeq v_k$ for all $k\leq n$ and all $n=0,1,\ldots\}.$  
\label{lemma_strictly_decreasing}
\end{lemma}
\begin{proof}
By Lemma \ref{lemma_ISCP_and_hereditary}, there is a countable set $F=\{u_n\in E^0-H\mid n=0,1,\ldots\}$ such that $ E^0-H=R(F).$ First, we create a set $D=\{w_n\in F\mid w_n\geq w_{n+1}$ for all $n=0,1,\ldots\}.$
Let $w_0=u_0.$ Continue inductively as follows. Assuming that $w_k$ has been defined for $k<n$, consider whether $w_{n-1}\geq u_n$ holds. If it does, then $w_n=u_n.$ If it does not, let $u\in E^0-H$ be a vertex such that $u_n\geq u, w_{n-1}\geq u$ and let $m$ be such that $u\geq u_m.$ Then let $w_n=u_m.$ By construction, $F\subseteq R(D).$ 

Let us now create a countable set $C\subseteq D$ with required properties. 
Let $v_0=w_0.$ Continue inductively as follows. Assuming that $v_k$ has been defined for $k<n$, consider whether $w_n\geq v_k$ holds for any $k<n.$ If it does, then let $v_n=w_{n-1}.$ If it does not, then $v_n=w_n.$ By construction, $C\subseteq D\subseteq R(C)$ thus $R(C)\subseteq R(D)\subseteq R(R(C))=R(C).$ This implies that $F\subseteq R(D)=R(C)$ and so $R(F)\subseteq R(R(C))=R(C).$ We also have that $C\subseteq D\subseteq F$ which implies that $R(C)\subseteq R(F).$ Thus, $R(C)=R(F)=E^0-H.$
\end{proof}

\begin{lemma} 
If $H\subseteq E^0$ is hereditary and saturated and $E^0-H$ is downwards directed and has ICSP, then $E^0-H=R(\alpha^0)$ for some path $\alpha$ such that if $\alpha$ is infinite, then it is strictly decreasing and if $\alpha$ is finite, then exactly one of the following holds for $v=\ra(\alpha).$   
\begin{enumerate}[\upshape(i)]
\item The vertex $v$ emits no edges to $E^0-H$.

\item The vertex $v$ is on a cycle which is extreme in $E^0-H.$   

\item The vertex $v$ is on an exclusive cycle. 
\end{enumerate}
\label{lemma_MT3_and_CSP}
\end{lemma}
\begin{proof}
By Lemma \ref{lemma_strictly_decreasing}, there is a countable set $C=\{v_n\in E^0-H\mid v_n\geq v_{n+1}$ and $v_{n+1}\ngeq v_k$ for all $k\leq n$ and all $n=0,1,\ldots\}$ such that $E^0-H=R(C).$ The properties of $C$ guarantee that if $v_n$ is on a cycle, then no closed path containing $v_n$ contains $v_m$ for any other $m\neq n.$

If $C$ is infinite, let $\alpha$ be a path consisting of paths which connect $v_n$ and $v_{n+1}.$ In this case, $\alpha$ is infinite and strictly decreasing. If $C$ is finite and $n$ is the smallest such that $v_n=v_m$ for all $m\geq n,$ then let $\alpha$ be a finite path connecting $v_i$ and $v_{i+1}$ for $i=0,\ldots, n-1.$ In  either case, $C\subseteq \alpha^0\subseteq R(C)=E^0-H$ and so $R(C)\subseteq R(\alpha^0)\subseteq R(R(C))=R(C).$ Thus, $R(\alpha^0)=E^0-H.$

In the second case, let $v=v_n=\ra(\alpha).$ If $v$ is not on a cycle, then $v$ emits no edges to $E^0-H$ (because if $e\in\so^{-1}(v)\cap \ra^{-1}(E^0-H)$, then  $\ra(e)\in R(v)$ so $e$ is on a cycle containing $v$). In this case (i) holds. If $v$ is on a cycle, say $c,$ and such cycle is not exclusive, we claim that $c$ is extreme in $E^0-H.$
Since all vertices of $c$ connect to $v$, $c^0\subseteq R(v)=E^0-H.$ As $c$ is not exclusive, $c$ has exits to $E^0-H.$ If $w\in E^0-H$ is the range of a path $p$ such that $p^0\subseteq E^0-H,$ $\so(p)\in c^0$ and $w\notin c^0,$ then $w\geq v,$ so there is a path $q$ with vertices in $E^0-H$ from $w$ to $v.$ This shows that $c$ is extreme in $E^0-H,$ so (ii) holds. If $v$ is on a cycle $c$ and $c$ is exclusive, (iii) holds.
\end{proof}

\begin{lemma}
The following conditions are equivalent for a hereditary and saturated set $H\subseteq E^0$ and $S\subseteq B_H.$ 
\begin{enumerate}[\upshape(1)]
\item $E^0-H$ is downwards directed and has ICSP and $S$ is either $B_H$ or $B_H-\{u\}$ where $u\in B_H$ is such that $E^0-H=R(u).$ 
\item The set of vertices of the quotient graph $E/(H,S)$ is downwards directed and has CSP. 
\end{enumerate}
\label{lemma_quotient_graph}
\end{lemma}
\begin{proof}
(1) $\Rightarrow$ (2). Assuming that (1) holds, let us consider two vertices of $E/(H,S).$ We have that either both are also vertices of $E$ or one (and only one of them by (1)) is of the form $u'$ for some infinite emitter $u\in B_H-S$ such that $E^0-H=R(u).$ In the first case, a lower bound in $E^0-H,$ which exists by the downwards directedness of $E^0-H,$ is a lower bound in the quotient graph as well. In the second case, both vertices are in $R(u),$ so $u'$ is a lower bound in the quotient graph. Since $E^0-H$ has ICSP, there is a countable set
$C\subseteq E^0-H$ such that $E^0-H\subseteq R(C).$ If $S=B_H,$ then the existence of $C$ shows that the quotient graph has CSP. If $B_H-S=\{u\},$ then $C\cup\{u'\}\subseteq (E/(H,S))^0=E^0-H\cup\{u'\}\subseteq R(C)\cup \{u'\}\subseteq R(C\cup\{u'\}).$ 

If (2) holds, consider two vertices in $E^0-H.$ If their lower bound in the quotient graph is a vertex of $E^0-H,$ this lower bound is a lower bound in $E^0-H$ also. If it is $u'$ for some $u\in B_H$, then these two vertices are in $R(u)$ in $E,$ so $u$ is their lower bound in $E^0-H.$ If $|B_H-S|>2,$ then the quotient graph has two sinks which would contradict the downwards directedness. Thus, either $B_H=S$ or $B_H-S=\{u\}$ for some $u\in B_H.$ In the first case, $E^0-H$ has ICSP because the existence of countable $C$ such that $C\subseteq (E/(H,S))^0=E^0-H\subseteq R(C)$ implies that $E^0-H$ has ICSP.
If $B_H-S=\{u\},$ we show that $R(u)=E^0-H.$ Since $u'$ is a sink in the quotient graph and $(E/(H,S))^0$ is downwards directed, $(E/(H,S))^0=R(u').$ Since the relation 
$v\geq u'$ in the quotient graph holds if and only if $v\geq u$ in $E$ for every $v\in E^0-H,$ we have that $E^0-H=R(u).$ This also shows that $E^0-H$ has ICSP since $\{u\}\subseteq E^0-H=R(u).$
\end{proof}

In the next example, we illustrate that $(E/(H,S))^0$ can be downwards directed and with CSP both when $S=B_H$ and when $S$ is such that $|B_H-S|=1.$ 

\begin{example}
Let $E$ be the following graph 
$\;\;\;\xymatrix{\bullet^v \ar@(ul,dl) 
\ar@{.} @/_1pc/ [r] _{\mbox{ } } \ar@/_/ [r] \ar [r] \ar@/^/ [r] \ar@/^1pc/ [r] &\bullet^w  }.$
Let $H=\{w\}$ so that $E^0-H=\{v\}=R(v),$ and $B_H=\{v\}.$
If $S=\emptyset,$ the quotient graph is 
$\;\;\;\xymatrix{\bullet^v \ar@(ul,dl) 
\ar[r] &\bullet^{v'}}$ and, if $S=\{v\},$ the quotient graph is $\;\;\;\xymatrix{\bullet^v \ar@(ul,dl)}.$ In both cases, $(E/(H,S))^0$ is downwards directed and with CSP. 
\label{example_inf_emitter_in_B_H}
\end{example}

\begin{theorem}
If $P$ is an ideal of $L_K(E),$
the following conditions are equivalent. 
\begin{enumerate}[\upshape(1)]
\item The ideal $P$ is graded primitive. 
\item The ideal $P$ is of the form $I(H, S)$ for an admissible pair $(H,S)$ such that $E^0-H$ is downwards directed and has ICSP and $S$ is either $B_H$ or $B_H-\{u\}$ where $u$ is a vertex of $B_H$ such that $E^0-H=R(u).$
\item Exactly one of the following holds.
\begin{enumerate}[\upshape(a)]
\item $P=I(H, B_H)$ where $H=E^0-R(\alpha^0)$ for a strictly decreasing infinite path $\alpha$ such that $R(v)\subsetneq R(\alpha^0)$ for all $v\in E^0-H.$ 

\item $P=I(H, B_H)$ where $H=E^0-R(v)$ for a vertex $v$  which emits no edges to $E^0-H.$ 

\item $P=I(H, S)$ where $H=E^0-R(v)$ for a vertex $v$ which is on a cycle extreme in $E^0-H$ and $S=B_H$ or $S=B_H-\{u\}$ for some $u\in B_H$ which is on the same cycle as $v$.  

\item  $P=I(H, S)$ where $H=E^0-R(v)$ for a vertex $v$ which is on an exclusive cycle and $S=B_H$ or $S=B_H-\{u\}$ for some $u\in B_H$ which is on the same cycle as $v$. 
\end{enumerate}
\end{enumerate}
\label{graded_primitive_ideal}
\end{theorem} 
\begin{proof}
$(1) \Leftrightarrow (2).$
If $P$ is a graded ideal, then $P=I(H,S)$ for some admissible pair $(H, S)$ and $L_K(E)/P$ is graded isomorphic to the Leavitt path algebra of the quotient graph $E/(H,S).$ Since $P$ is a graded primitive ideal if and only if $L_K(E)/P$ is a graded primitive ring, the equivalence of (1) and (2) follows directly from Lemma \ref{lemma_quotient_graph} and Theorem \ref{graded_primitive_LPA}.  

(1) and $(2) \Rightarrow (3).$ By Lemma \ref{lemma_MT3_and_CSP}, 
if (2) holds, then $E^0-H=R(\alpha^0)$ where $\alpha$ satisfies one of the two cases from Lemma \ref{lemma_MT3_and_CSP}. If $\alpha$ is infinite, then it is strictly decreasing. If $R(v)=R(\alpha^0)$ for some $v\in E^0-H,$ we can replace $\alpha$ by $v$ and consider the case when $\alpha$ is finite by the proof of Lemma \ref{lemma_MT3_and_CSP}. Thus, we can assume that $R(v)\subsetneq R(\alpha^0)$ for all $v\in E^0-H$ and we show that $S=B_H$ in this case. If $v\in B_H-S,$ and $u\in R(\alpha^0)-R(v),$ then $v'$ is a sink and $u\notin R(v')$ in the quotient graph so (2) fails by Lemma \ref{lemma_quotient_graph}. Thus, such $v$ cannot exist and we have that $B_H=S.$ Hence, (a) holds.

If $\alpha$ is finite and $v=\ra(\alpha),$ then exactly one of the conditions (i), (ii), and (iii) of Lemma \ref{lemma_MT3_and_CSP} holds. If it is (i), then $v$ is not on the same cycle as $u$ for any $u\in E^0-H$ and we claim that $S=B_H.$ If $u\in B_H-S,$ for some $u\in E^0-H,$ then $u'$ is a sink and $v\notin R(u')$ in the quotient graph, so (2) fails by Lemma \ref{lemma_quotient_graph}. Hence, no such $u$ can exist. Thus, $S=B_H$ and (b) holds. If (ii) holds, note that if $S=B_H-\{u\}$, then $u$ and $v$ are on the same cycle since $R(v)=E^0-H$ and $R(u)=E^0-H$ by (2), so (c) holds. Similarly, if (iii) holds, then (d) holds. 
 
$(3) \Rightarrow (2).$ Since $E^0-H$ is of the form $R(\alpha^0)$ for a (finite or infinite) path $\alpha$ with vertices in $E^0-H$ in all four cases, we have that $E^0-H$ is downwards directed and has ICSP. Also, in all four cases, the set $S$ has the required form. 
\end{proof}

For the next corollary, we recall \cite[Theorem 4.7]{Gene_Bell_Ranga} stating that $L_K(E)$ is a primitive ring if and only if $E^0$ is downwards directed and has Condition (L) and CSP. 

\begin{corollary}
A graded primitive ideal $P=I(H,S)$ of $L_K(E)$ is not primitive if and only if there is an exclusive cycle $c$ such that $H=E^0-R(c^0)$ and $S=B_H.$    
\label{corollary_graded_primitive_and_primitive} 
\end{corollary}
\begin{proof}
If $P=I(H,S)$ is graded primitive with $H=E^0-R(c^0)$ and $S=B_H$ for an exclusive cycle $c,$ then $c$ is a cycle without exits in the quotient graph. So, the quotient graph does not have Condition (L). By \cite[Theorem 4.7]{Gene_Bell_Ranga}, $L_K(E)/P$ is not a primitive ring. Thus, $P$ is not primitive. 

To show the converse, assume that $P=I(H,S)$ is a graded primitive ideal satisfying a condition of part (3) of Theorem \ref{graded_primitive_ideal} except part (3d) in the case when $S=B_H.$ Since part (2) of Theorem \ref{graded_primitive_ideal} holds in this case, $(E/(H,S))^0$ is downwards directed and has CSP by Lemma \ref{lemma_quotient_graph}. We show that $(E/(H,S))^0$  has Condition (L) by considering a cycle $c$ with $c^0\subseteq E^0-H.$ We show that either there is a vertex $v\in E^0-H$ such that $v\notin c^0\subseteq R(v)$ holds, or that $c$ is extreme in $E^0-H$, or that there is $u\in B_H$ such that $u'\notin c^0\subseteq R(u')$ holds in the quotient graph.  

If (3a) of Theorem \ref{graded_primitive_ideal} holds, let $\alpha$ be an infinite path as in condition (3a). Since $c^0$ is finite and $\alpha^0$ infinite, let $w\in\alpha^0$ be the largest (with respect to $\geq$) such that $u\notin c^0$ for any $u\in \alpha^0$ for which  $w\geq u.$ As $c^0\subseteq R(\alpha^0),$ there is $u\in \alpha^0$ such that $c^0\subseteq R(u).$ By taking $v\in \alpha^0$ such that $u\geq v$ and $w\geq v,$ we have that $v\notin c^0\subseteq R(u)\subseteq R(v).$   

If (3b) of Theorem  \ref{graded_primitive_ideal} holds for some $v\in E^0-H,$ then $v\notin c^0$ as $v$ does not emit edges to $E^0-H.$ So, $v\notin c^0\subseteq  E^0-H=R(v).$  

If (3c) of Theorem \ref{graded_primitive_ideal} holds for $v\in E^0-H,$ then if $v\notin c^0,$ we have that $v\notin c^0\subseteq  E^0-H=R(v),$ and if $v\in c^0,$ then $c$ is extreme in $E^0-H,$ so $c$ has an exit with the range in $E^0-H.$

If (3d) holds and $S\neq B_H$, let $d$ be the exclusive cycle from (3d) and let $u$ be the only element of $B_H-S.$ Since $u\in d^0$ and $d$ is exclusive, if $c\neq d,$ then $u\notin c^0\subseteq R(u).$ If $c=d,$ then $u'$ is not in $c^0$ in the quotient graph (since it is a sink) and $c^0\subseteq R(u').$ 

This shows that the quotient graph has Condition (L). Hence, $L_K(E)/P$ is a primitive ring by \cite[Theorem 4.7]{Gene_Bell_Ranga}. Thus, $P$ is a primitive ideal.  
\end{proof}

As a corollary, we obtain the following description of all (not necessarily graded) primitive ideals. 

\begin{corollary}
If $P$ is an ideal of $L_K(E),$ the following conditions are equivalent. 
\begin{enumerate}[\upshape(1)]
\item The ideal $P$ is primitive.
\item Exactly one of the following holds.
\begin{enumerate}[\upshape(a)]
\item $P=I(H, B_H)+\langle f(c)\rangle$ where $c$ is an exclusive cycle, $E ^0-H=R(c^0)$ and $f$ is an irreducible polynomial in $K[x, x ^{-1}].$

\item Exactly one of the following condition of Theorem \ref{graded_primitive_ideal} holds: (3a), (3b), (3c), or (3d) with $S$ such that $|B_H-S|=1.$  
\end{enumerate}
\end{enumerate}
\label{corollary_primitive_ideal}
\end{corollary}
\begin{proof}
If $P$ is non-graded, then (1) is equivalent to (2a) by \cite[Theorems 3.12 and 4.3]{Ranga_prime}.
If $P$ is graded, (1) is equivalent to (2b) by Corollary \ref{corollary_graded_primitive_and_primitive}. 
\end{proof}

Using Theorem \ref{graded_primitive_ideal} and Corollary \ref{corollary_primitive_ideal}, it is direct to construct a graded primitive ideal which is not primitive. 
For example,  let $E$ be the graph  
$\;\;\;\xymatrix{{\bullet}\ar@(lu,ld)}.$ For   $H=S=\emptyset$, $0=I(H,S)$ a graded primitive ideal, but $E$ does not have Condition (L), so $0$ is not primitive.

\section{Complete characterization of graded irreducible representations}\label{section_graded_irreducible}

\subsection{Defining a new type of graded irreducible representations.} \label{subsection_defining_N_c}
Throughout this section, $E$ is a graph with an exclusive cycle $c$ and $v$ is a vertex of $c.$ 
We define a branching system $X^v,$ or $X$ for short if it is clear which vertex we selected, as follows.

Let $Y=Y^v$ be the set of the basis elements of the path algebra $P_K(\widehat{E})$ of the extended graph $\widehat{E}$ which have the form $pq^*$ where $p,q$ are paths with $\ra(p)=\ra(q)\in c^0$ and $\so(q)=v.$ Note that the path $q$ in an element $pq^*\in Y$ is completely in $c$ because $c$ is exclusive.

We say that $pq^*\in Y$ is {\em not reduced} if $p$ and $q$ have positive lengths and they end with the same edge $e$. This defines a function on $Y$ which denote by $\red,$ and call the {\em reduction function}, by 
\[\red(ee^*q^*)=q^*\mbox{ and }\red(pee^*q^*)=\red(pq^*)\]
for an edge $e$ in $c$ and $ee^*q^*, pee^*q^*\in Y.$ We say that $pq^*\in Y$ is {\em reduced} if $\red(pq^*)=pq^*$.

Note that if $\red(pq^*)=rs^*$ then $\so(p)=\so(r)$ and $\so(q)=\so(s)=v.$ The following lemma shows a property of the reduction function which we need in Lemma \ref{lemma_branching_system}. 

\begin{lemma}
If $pq^*\in Y$ and $e\in E^1$ is such that $\so(p)=\ra(e),$ then $\red(e\red(pq^*))=\red(epq^*).$ 
\label{lemma_reduction}
\end{lemma}
\begin{proof}
We prove the claim by induction on $|p|.$ If $|p|=0,$ then $p$ is the vertex $\ra(e)$ and $\red(pq^*)=q^*$ so $\red(e\red(pq^*))=\red(eq^*)=\red(epq^*).$ Assuming the induction hypothesis, let $p=rf$ for some $f\in E^1$ and path $r$ in which case $\so(r)=\so(p)=\ra(e).$ If $pq^*$ is reduced, the claim trivially holds since $\red(pq^*)=pq^*.$ If $pq^*$ is not reduced, then $q=sf$ for some path $s$ and, by induction hypothesis, 
\[\red(e\red(pq^*))=\red(e\red(rff^*s^*))=
\red(e\red(rs^*))=\red(ers^*)=
\red(erff^*s^*)=\red(epq^*).\]
\end{proof}

If $|p|$ denotes the length of $p,$ let us define the degree function on $Y$ by $\deg(pq^*)=|p|-|q|.$ Since reducing does not change the degree of $pq^*\in Y,$ we have that $\deg(pq^*)=\deg(\red(pq^*)).$ We define $X\subseteq Y$ in the following definition. Lemma \ref{lemma_branching_system} shows that $X$ is a branching system.

\begin{definition} Let $X=\{ pq^*\in Y\mid \, \red(pq^*)=pq^*\}$ be the set of elements of $Y$ which are reduced. For $w\in E^0$ and $e\in E^1,$ let 
\[X_w=\{pq^*\in X\mid \so(p)=w\},\;\;\;\;X_e=\{pq^*\in X \mid pq^*=\red(ers^*)\mbox{ for some }ers^*\in Y\},\]\[\sigma_e: X_{\ra(e)}\to X_e\;\mbox{ be given by }\;pq^*\mapsto \red(epq^*),\mbox{ and }\deg: X\to \Zset\mbox{ be given by }pq^*\mapsto |p|-|q|.\]
\label{definition_branching_for_N_c}
\end{definition}

\begin{lemma} The set $X$ is a graded branching system which is perfect and saturated. 
\label{lemma_branching_system}
\end{lemma}
\begin{proof}
In the rest of the proof, let $u,w\in E^0$ and $e,f\in E^1.$ By Definition \ref{definition_branching_for_N_c}, $X_u\cap X_w=\emptyset$ if $u\neq w.$
If $e\neq f$ and $e,f$ are not it $c$, $X_e\cap X_f=\emptyset$ and $X_e\subseteq X_{\so(e)}$ also by Definition \ref{definition_branching_for_N_c}. If $e$ is in $c,$ we claim that $X_e=X_{\so(e)}.$ If $pq^*\in X_e,$ then $pq^*=\red(ers^*),$ so $\so(p)=\so(er)=\so(e).$ Hence, $pq^*\in X_{\so(e)}.$ Conversely, if $pq^*\in X_{\so(e)},$ we consider the cases when the length of $p$ is zero and when it is positive. 
In the first case, $pq^*=q^*=\red(ee^*q^*),$ so $pq^*\in X_e.$ In the second case, $\so(p)=\so(e)$ which implies that $p$ starts with $e$ since $c$ is exclusive. Hence, $pq^*=erq^*$ for some path $r.$ Since $pq^*$ is reduced, we have that $pq^*=\red(pq^*)=\red(erq^*)$ showing that $pq^*\in X_e.$ 

The relation $X_e=X_{\so(e)}$ for any edge $e$ of $c,$ implies that $X_e\cap X_f=\emptyset$ for any $e\neq f\in E^1.$ 

For $e\in E^1,$ we claim that the map $\sigma_e^{-1}: X_e\to X_{\ra(e)}$  given by 
\[\sigma_e^{-1}:\red(epq^*)\mapsto \red(pq^*)\]
is the inverse of $\sigma_e.$
If $pq^*\in X_{\ra(e)}$ then $\sigma_e^{-1}\sigma_e(pq^*)=\sigma_e^{-1}(\red(epq^*))=\red(pq^*)=pq^*$ and if $\red(epq^*)\in X_e,$ then $\sigma_e\sigma_e^{-1}(\red(epq^*))=\sigma_e(\red(pq^*))=\red(e\red(pq^*))=\red(epq^*)$ where the last equality holds by Lemma \ref{lemma_reduction}.   

For $w\in E^0$ which is not a sink, let $W=\bigcup_{e\in\so^{-1}(w)}X_e$. Since $X_e\subseteq X_{\so(e)}$ for every $e,$ $W\subseteq X_w.$  
We show $X_w\subseteq W.$ If $w$ is a vertex which is not in $R(c^0),$ then $\emptyset=X_w\subseteq W=\emptyset.$  
If $w\in R(c^0)-c^0,$ then any $pq^*\in X_w$ is of the form $erq^*$ for some path $r$ and some $e\in\so^{-1}(w).$ Since $erq^*=\red(erq^*)\in X_e,$ we have that $pq^*\in X_e$ and so $X_w\subseteq W.$ If $w\in c^0,$ then $\so^{-1}(w)$ contains an edge, say $e,$ which is in $c.$ As we showed before, $X_e=X_{\so(e)}.$ So, $X_w=X_e\subseteq W.$  

This shows that $X$ is a perfect branching system. Since $pq^*\in X$ is in $X_{\so(p)},$ $X$ is saturated. The branching system $X$ is graded as, for $pq^*\in X,$
\[\deg(\sigma_e(pq^*))=\deg(\red(epq^*))=\deg(epq^*)=|ep|-|q|=|p|-|q|+1=\deg(pq^*)+1.\]
\end{proof}

By Lemma \ref{lemma_branching_system}, $M(X)$ is the $K$-linear span with basis $X$ and with the $L_K(E)$-module structure induced by 
\begin{center}
\begin{tabular}{l}
$w\cdot pq^*=pq^*$  if $\so(p)=w$ and $w\cdot pq^*=0$ otherwise, \\
$e\cdot pq^*=\sigma_e(pq^*)=\red(epq^*)$ if $\so(p)=\ra(e)$ and $e\cdot pq^*=0$ otherwise, and \\
$e^*\cdot pq^*=\sigma_e^{-1}(pq^*)=\red(rs^*)$  if $pq^*=\red(ers^*)$ for some $ers^*\in Y$ and  $e^*\cdot pq^*=0$ otherwise 
\end{tabular}
\end{center}
for $w\in E^0, e\in E^1,$ and $pq^*\in X.$
We use $\N^v_c$ to denote $M(X).$ If it is clear that $X=X^v,$ we use only $\N_c$ for $\N^v_c.$ 
 
The following lemma focuses on a property of $\N_c$ which we need in Lemma \ref{lemma_describing_homogeneous_element}. 
\begin{lemma}
If $pq^*\in Y$ and $\so(p)\in c^0,$ then $p^*\cdot \red(pq^*)=q^*.$ 
\label{lemma_ghost_action} 
\end{lemma}
\begin{proof}
We use induction on the length $|p|.$ If $p$ is a vertex $w$ in $c^0,$ then $w=\ra(q)$ and $p^*\cdot \red(pq^*)=w\cdot \red(q^*)=w\cdot q^*=q^*.$ Assuming the induction hypothesis for $pq^*\in Y$, let us show the claim for $epq^*\in Y$ where $e\in E^1$ is with $\so(e)\in c^0.$ In this case, 
\[(ep)^*\cdot \red(epq^*)=p^*\cdot(e^*\cdot \red(epq^*))=p^*\cdot \sigma_e^{-1}(\red(epq^*))=p^*\cdot \red(pq^*)=q^*.\]   
\end{proof}

We consider the form of a homogeneous element of $\N_c$ next. Let $a=\sum_{i=1}^n k_i\red(p_iq_i^*)$ be a homogeneous element of $\N_c.$  Since each $p_i$ ends in $c,$ let $p_i=t_ic_i$ where $\ra(t_i)\in c^0$ and every other vertex of $t_i$ is not in $c^0.$ This enables us to partition the set $\{1,\ldots, n\}$ into $I_j, j=1,\ldots k$ for some $k\in\{1,\ldots n\},$ such that $l\in I_j$ if $t_l=t_j.$ So, we let 
\[a=\sum_{j=1}^k a_j\mbox{ where }a_j=\sum_{l\in I_j}k_l\red(p_lq_l^*)=\sum_{l\in I_j}k_l\red(t_jc_lq_l^*).\] 
Since $a$ is homogeneous, $a_j$ is homogeneous for every $j=1,\ldots,k.$ This implies that $\deg(\red(p_lq_l^*))=\deg(\red(t_jc_lq_l^*))=|t_j|+|c_l|-|q_l|$ is the same for every $l\in I_j.$ In particular, $|c_l|-|q_l|=|c_j|-|q_j|$ for every $l\in I_j.$ 

If $|c_j|-|q_j|\geq 0,$ let $d_j$ be the part of $c$ ending at $v$ of length $|c_j|-|q_j|.$ If $|c_j|-|q_j|<0,$ 
let $f_j$ be the part of $c$ starting at $v$ of length $|q_j|-|c_j|$ and let $d_j=f_j^*.$ So, in either case, $\red(c_lq_l^*)=\red(c_jq_j^*)=d_j$ and
$\red(p_lq_l^*)=\red(t_jc_lq^*_l)=t_j\red(c_lq^*_l)=t_jd_j$ for any $l\in I_j.$  Denote $\sum_{l\in I_j} k_l$ by $k_j$ so that 
\[a=\sum_{j=1}^k a_j\;\mbox { and }\;a_j=\sum_{l\in I_j} k_l t_jd_j=k_j t_jd_j\mbox{ for all }j=1,\ldots, k.\]
We use this notation for the following lemma. 

\begin{lemma}
If $a=\sum_{i=1}^n k_i\red(p_iq_i^*)=\sum_{j=1}^k a_j\in \N_c$ is homogeneous with $a_j=k_jt_jd_j,$ then 
$p_j^*\cdot a=k_jq_j^*$ for $j=1,\ldots, k.$ Thus, if $p_j^*\cdot a= 0$ for all $j=1,\ldots k,$ then $a=0.$ 
\label{lemma_describing_homogeneous_element} 
\end{lemma}
\begin{proof}
Note that
$p_j^*p_l=0$ for $l\notin I_j$ holds in $L_K(E),$ so that $p_j^*\cdot a_l=0$ for $l\notin I_j.$ Thus, 
\[p_j^*\cdot a=p_j^*\cdot a_j+p_j^*\cdot \sum_{l=1, l\neq j}^k a_l= p_j^*\cdot a_j+0=p_j^*\cdot a_j=c_j^*t_j^*\cdot k_jt_jd_j=k_j\, c_j^*\cdot d_j=k_j\, c_j^*\cdot \red(c_jq_j^*)=k_jq_j^*\]
where the last equality holds by Lemma \ref{lemma_ghost_action}.
Since $q_j^*$ is a basis element of $N_c,$ it is nonzero and, if $k_jq_j^*=p_j^*\cdot a=0,$ then $k_j=0.$ Thus, if $p_j^*\cdot a=0$ for all $j=1,\ldots k,$ then $k_j=0$ for all $j=1,\ldots, k$ which implies that $a=\sum_{j=1}^k a_j=0.$
\end{proof}

Lemma \ref{lemma_describing_homogeneous_element} is key to proving part (2) of the following proposition. 

\begin{proposition}
Let $H=E^0-R(c^0).$ 
\begin{enumerate}[\upshape(1)]
\item $\ann(\N_c)=I(H, B_H).$

\item $\N_c$ is graded simple and not simple. 

\item If $c$ has more than one vertex and if $w\neq v$ are vertices of $c,$ then $\N_c^v$ and $\N_c^w(n)$ are graded isomorphic where $n$ is the length of the path from $v$ to $w$ in $c$. 

\item If $d$ is an exclusive cycle and $w\in d^0$, the modules $\N^v_c$ and $\N^w_d$ are graded isomorphic if and only if $c=d$ and $v=w$. The modules $\N^v_c$ and $\N^w_d$ are isomorphic if and only if $c=d.$ 

\item The module $\N_c$ is not isomorphic to any of the Chen modules (see section \ref{subsection_list_of_Chen_modules}). Moreover, the annihilator
of $\N_c$ is different from the annihilator of any of the Chen modules.
\end{enumerate}
\label{proposition_N_c}
\end{proposition}
\begin{proof}
The set $H$ is hereditary and saturated since it has the form $E^0-R(V)$ for some $V\subseteq E^0.$

(1) Note that $\ann(N_c)$ is graded since $N_c$ is graded. Thus, $\ann(N_c)=I(G, S)$ for some admissible pair $(G, S)$. If $w\in H,$ then $w\cdot pq^*=0$ for all $pq^*\in X$ since $\so(p)\in E^0-H,$ so $w\cdot a=0$ for all $a\in N_c.$ Thus, $H\subseteq G.$  If $w\notin H,$ then $w\in R(c^0)$ and if $p$ is a path from $w$ to $v,$ then $p\in \N_c$ and $w\cdot p=p\neq 0,$ so $w\notin G.$ Thus, $H=G.$ If $w\in B_H,$ we claim that $w^H\cdot pq^*=0$ for any $pq^*\in X.$
Since $w^H\cdot pq^*=0$ for any $pq^*\in X-X_w,$ it is sufficient to consider $w^H\cdot pq^*$ for $pq^*\in X_w.$
We consider the cases $w\notin c^0$ and $w\in c^0.$ 
In the first case, the length of $p$ is positive. Let $e$ be the first edge of $p$ and $p=er$ for a path $r$ so that 
$ee^*\cdot pq^*=e\cdot \sigma_e^{-1}(pq^*)=e\cdot rq^*=\red(erq^*)=\red(pq^*)=pq^*.$ As $e\in \so^{-1}(w)\cap \ra^{-1}(E^0-H),$ we have that 
\[w^H\cdot pq^*=(w-\sum_{f\in \so^{-1}(w)\cap \ra^{-1}(E^0-H)}ff^*)\cdot pq^*=w\cdot pq^*-ee^*\cdot pq^*=pq^*-pq^*=0.\]
In the second case, the set $\so^{-1}(w)\cap \ra^{-1}(E^0-H)$ consists of one edge since $c$ is exclusive. Let $e$ be that edge. Since $e$ is in $c,$ $X_w=X_e,$ and so $pq^*\in X_e.$ Let $pq^*=\red(ers^*).$ Then $ee^*\cdot pq^*=e\cdot\sigma_e^{-1}(\red(ers^*))=e\cdot\red(rs^*)=\red(e\red(rs^*))=\red(ers^*)$ where the last equality holds by Lemma \ref{lemma_reduction}. As $\red(ers^*)=pq^*,$ we have that 
$ee^*\cdot pq^*=pq^*$ in this case also. Hence, 
$w^H\cdot pq^*=(w-ee^*)\cdot pq^*=w\cdot pq^*-ee^*\cdot pq^*=pq^*-pq^*=0.$ This shows that $w^H\in \ann(\N_c)$ and so $w\in S.$

(2) If $M$ is a nontrivial graded submodule of $\N_c,$ we show that $M=\N_c$ by showing that $v$ is in $M.$ Since $M$ is nontrivial and graded, there is a nonzero homogeneous element $a$ in $M.$ Using the notation we introduced earlier, let  $a=\sum_{i=1}^n k_i\red(p_iq_i^*)=\sum_{j=1}^k a_j$ where $a_j=k_jt_jd_j$ for $j=1,\ldots, k.$ Since $a\neq 0,$ there is $j\in \{1,\ldots, k\}$ such that $k_jq_j^*=p_j^*\cdot a\neq 0$
by Lemma \ref{lemma_describing_homogeneous_element}. This implies that $k_j\neq 0$ and so  
\[v=k_j^{-1}  k_jv = k_j^{-1}  k_j \red(q_jq_j^*)=k_j^{-1}q_j\cdot k_jq_j^*=k_j^{-1}q_j\cdot (p_j^*\cdot a)=k_j^{-1}p_jq_j^*\cdot a\in M.\]

This also shows that $\ann(\N_c)=I(H, B_H)$ is graded primitive. By Corollary \ref{corollary_graded_primitive_and_primitive}, $\ann(\N_c)$ is not primitive, so $\N_c$ is not simple. 

(3) If $c_{vw}$ is the path from $v$ to $w$ in $c,$ let  $f_{vw}: X^w\to X^v$ be given by $pq^*\mapsto \red(pq^*c_{vw}^*).$ Then 
$f(X^w_u)=X^v_u$ for $u\in E^0$ and  $f(X^w_e)=X^v_e$ for  $e\in E^1.$
If $pq^*\in X^w_{\ra(e)}$ for some $e\in E^1,$ then 
\[f\sigma_e^w(pq^*)=f(\red(epq^*))=\red(\red(epq^*)c_{vw}^*)=
\red(epq^*c_{vw}^*)=\sigma_e^v(pq^*c_{vw}^*)=\sigma_e^v(f(pq^*)).\]
This shows that the branching systems are isomorphic in the sense of \cite[Section 5]{Chen} and, hence, $f$ extends to an isomorphism of $\N^w_c$ and $\N^v_c.$ If $|c_{vw}|=n,$ $f$ maps an element of the $m+n$-component of $\N^w_c$ to an element of the $m$-component of $\N_c^v$, so $f$ is a graded isomorphism of $\N^w_c(n)$ and $\N_c^v.$  

(4) If $c\neq d,$ then $R(v)\neq R(w)$ because the cycles $c$ and $d$ are exclusive, so $E^0-R(v)\neq E^0-R(w).$ Thus, $\ann(\N^v_c)\neq \ann(\N^w_d)$ which implies that $\N^v_c$ and $\N^w_d$ are not isomorphic.
If $c=d$ but $v\neq w,$ then $\N^v_c$ is graded isomorphic to $\N^w_d(n)$ for some nonzero $n$ by part (3). This implies that $\N^v_c$ and $\N^w_d$ are isomorphic (but not graded isomorphic). The converse follows by part (3).

(5) The graded module $\N_c$ is not isomorphic to any of the Chen modules since they are simple and, by part (2), $\N_c$ is not. Next, we show that $\ann(\N_c)$ is different from the annihilator of any other Chen module. If an infinite path $\alpha$ is rational, $\ann(\V^f_{[\alpha]})$
is not graded, so it is different from $\ann(\N_c).$ 
The same argument applies for $\ann(\V_{[\alpha]})$ if $\alpha$ is rational and it ends in an exclusive cycle. If $\alpha$ is rational and it ends in a cycle $d$ which is not exclusive, then $d\neq c$ and either $d^0\subseteq H$ or $v\notin R(\alpha^0).$ In any case, $\ann(\V_{[\alpha]})\neq\ann(\N_c).$ If $\alpha$ is irrational, then $v\notin \alpha^0$ and $R(v)\neq R(\alpha^0)$ since $c$ is exclusive. Thus, $\ann(\V_{[\alpha]})\neq \ann(\N_c).$  
If $E$ has a sink $w$, $\ann(\N_w)$ satisfies (3b) and $\ann(\N_c)$ satisfies (3d) of Theorem \ref{graded_primitive_ideal}. If $u$ is an infinite emitter which is not in $c$ then $R(v)\neq R(u),$ so $\ann(\N_c)\neq \ann(\N_u).$  
If $u$ is an infinite emitter which is in $c,$ then $u$ emits only one edge to $E^0-H$ since $c$ is without exits in $E^0-H.$ Thus, $u\in B_H$ and $\ann(\N_c)=I(H, B_H)\neq I(H, B_H-\{u\})=\ann(\N_{u\in B_H}).$ 
\end{proof}

\begin{remark}
Let $X^c=\bigcup_{v\in c^0} X^v.$ Lemmas \ref{lemma_reduction} and \ref{lemma_branching_system} and part (1) of Proposition \ref{proposition_N_c} still hold for $X^c$ and $M(X^c)$. However,
since  $M(X^c)=\bigoplus_{v\in c^0}M(X^v)=\bigoplus_{v\in c^0}\N_c^v,$
part (2) of Proposition \ref{proposition_N_c} holds if and only if $c^0$ consists of only one vertex.
\end{remark}

\subsection{Characterization of graded irreducible representations}\label{subsection_characterization}
In light of part (2) of Proposition \ref{proposition_N_c}, we introduce the term {\em graded Chen module} to denote Chen modules which are graded and modules of the type $\N_c.$ This list of modules is exhaustive in the sense detailed in Theorem \ref{graded_simple}. 
 
\begin{theorem}
If $M$ is a graded simple $L_K(E)$-module then the annihilator of $M$ is the same as the annihilator of exactly one of the graded Chen modules below.
\begin{enumerate}[\upshape(1)]
\item {\em The strictly-decreasing-infinite-path type:} $\V_{[\alpha]}$ where $\alpha$ is a strictly decreasing infinite path such that $ R(v)\subsetneq R(\alpha^0)$ for any $v\in R(\alpha^0).$

\item {\em The relative-sink types:} $\N_v$ where $v$ is a sink or $\N_{v\emptyset}$ where $v$ is an infinite emitter with $\ra(\so^{-1}(v))\cap R(v)=\emptyset.$  

\item  {\em The extreme-cycle types:} $\V_{[\alpha]}$ where $\alpha$ is an irrational path with $\alpha^0$ on cycles extreme in $R(\alpha^0),$  $\N_{v\infty},$ or $\N_{v\in B_H}$ where $v$ is an infinite emitter with $|\ra(\so^{-1}(v))\cap R(v)|>1.$    

\item  {\em The exclusive-cycle types:} $\N_c$ where $c$ is an exclusive cycle or $\N_{v\in B_H}$ where $v$ is an infinite emitter with $|\ra(\so^{-1}(v))\cap R(v)|=1.$ 
\end{enumerate}
\label{graded_simple} 
\end{theorem}
\begin{proof}
Assume that $M$ is a graded simple module and let $P=I(H, S)=\ann(M)$. Then $P$ is graded primitive, so exactly one  condition of part (3) of Theorem \ref{graded_primitive_ideal} holds. If it is (3a), then $\ann(\V_{[\alpha]})=P=\ann(M)$ for a path $\alpha$ as in (3a). If it is (3b), then $P=\ann(\N_v)$ or $P=\ann(\N_{v\emptyset})$ for a vertex $v$ as in (3b). 
If it is (3d) and $S=B_H,$ then $P=\ann(\N_c^v)$ for a vertex $v$ as in (3d). If it is (3d) and $|B_H-S|=1,$ then $B_H-S=\{u\}$ for some $u\in B_H$ as in (3d) and  $P=\ann(\N_{u\in B_H}).$

If $P$ is as in part (3c) of Theorem \ref{graded_primitive_ideal}, then there is $v$ on a cycle $c$ extreme in $E^0-H$ such that $E^0-H=R(v).$ If $S=B_H,$ then let $d$ be another cycle such that $c^0\cap d^0$ is nonempty ($c$ has exits which connect back to $c,$ so there is at least one such cycle $d$). The path $\alpha=cdccddcccddd\ldots$ is irrational and $R(v)=R(\alpha^0),$ so that $H=E^0-R(v)=E^0-R(\alpha^0).$
Since $\alpha$ does not end in an exclusive cycle, we have that $\ann(\V_{[\alpha]})=I(H, B_H)=I(H,S)=P.$ In addition, if there is also an infinite emitter $u$ such that $R(u)=R(v)$ and  $|\ra(\so^{-1}(v))\cap R(v)|=\infty,$ then we also have that $\ann(\N_{u\infty})=I(H, B_H)=P.$ If $S=B_H-\{u\}$ for some $u,$ then there is a cycle which contains both $u$ and $v$ by Theorem \ref{graded_primitive_ideal}, so that $R(u)=R(v)=E^0-H$ and $\ann(\N_{u\in B_H})=I(H, B_H-\{u\})=P.$
\end{proof}
 
By Example \ref{example_more_than_one_Chen}, it is possible that 
two graded Chen modules of different types have the same annihilator $I(H,S)$. By the proof of Theorem \ref{graded_simple}, this can happen only when one of the Chen modules is of the type $\V_{[\alpha]}$ where $\alpha$ is an infinite irrational path with vertices on cycles which are extreme in $E^0-H$ and the other is of the type $\N_{v\infty}$ for some infinite emitter $v$ with $|\ra(\so^{-1}(v))\cap R(v)|=\infty$ and such that $R(v)=E^0-H.$ By \cite[Proposition 2.8]{Ranga_simple_modules}, the modules of the form $\V_{[\alpha]}$ and $\N_{v\infty}$ are not isomorphic.

\section{Discussion on some results from   \texorpdfstring{\cite{Ranga_prime}}{TEXT} and \texorpdfstring{\cite{Roozbeh_Ranga}}{TEXT}}\label{section_corrections}

\subsection{\texorpdfstring{\cite[{\bf Theorem 4.3}]{Ranga_prime}}{TEXT}}
\label{subsection_Ranga}
\cite[Theorem 4.3]{Ranga_prime} states that an ideal $P$ of $L_K(E)$ with $H=P\cap E^0$ is primitive if and only if  
$P$ satisfies one of the following. 
\begin{enumerate}[\upshape(1')]
\item $P$ is a non-graded prime ideal. 
 
\item $P$ is a graded prime ideal of the form $I(H, B_H -\{u\})$ for some $u\in B_H.$ 

\item $P$ is a graded ideal of the form $I(H, B_H)$ and $E ^0-H$ is downwards directed, has Condition (L), and CSP holds. 
\end{enumerate}

We show that there is a graded ideal $P$ such that (3') holds but $P$ is not primitive.

In \cite[Proposition 3.5]{Gene_Bell_Ranga},  an acyclic graph, $F$ in the example below, was exhibited. The graph $F$ is such that $F^0$ is downwards directed but it does not have CSP. We add a sink $w$ to $F$ and an edge from each vertex of $F^0$ to $w$ to obtain a graph $E$ such that $E^0=R(w).$ Thus, $E^0$ has CSP and, hence, $E^0-\{w\}$ has CSP. However, $E^0-\{w\}$ does not have ICSP.

\begin{example} Let $F(\Rset)$ denote the set of finite subsets of the set of real numbers $\Rset.$ Let $F$ be the graph with vertices 
$v_X$ for $X\in F(\Rset)$ and edges $e_{XY}$ for $X,Y\in F(\Rset)$ such that $X\subsetneq Y$ and with the source and range maps given by $\so(e_{XY})=v_X$ and $\ra(e_{XY})=v_Y.$  

Let $E$ be the graph with vertices $F^0\cup \{w\},$ with edges $F^1\cup \{e_X\mid X\in  F(\Rset)\},$ with the source and the range maps the same as in $F$ on $F^1,$ and with $\so(e_X)=v_X$ and $\ra(e_X)=w.$

The set $H=\{w\}$ is hereditary and saturated with $B_H=\emptyset.$ The set $F^0=E^0-H$ is downwards directed, has Condition (L), and $\{w\}$ is such that $F^0\subseteq R(\{w\}),$ so CSP holds for $F^0.$ Hence, (3') holds for $P=I(\{w\}, \emptyset)$. 
However, $F^0$ does not have ICSP because no countable set of vertices $C$ such that $F^0\subseteq R(C)$ can be found within $F^0$ (as it was pointed out in \cite{Gene_Bell_Ranga}). Indeed, assuming that such $C$ exists, we would have that every finite subset of $\Rset$ is contained in a fixed countable subset $A$ of $\Rset.$ This would imply that $\Rset-A$ does not have any finite subsets which is a contradiction since $\Rset-A$ is not empty. 
This shows that $E^0-H=(E/(H, \emptyset))^0$ does not have CSP, so $P$ is not primitive by \cite[Theorem 4.7]{Gene_Bell_Ranga}. Note that $P$ is not primitive also by 
Theorem \ref{graded_primitive_ideal} and Corollary \ref{corollary_primitive_ideal}.  
\label{example_CSP_not_ISCP}
\end{example}

If $(H,S)$ is an admissible pair of a graph $E,$ then the assumption that $E^0-H$ has CSP does not imply that the quotient graph $E/(H,S)$ also has CSP. Lemma \ref{lemma_quotient_graph} shows that one needs to require that $E^0-H$ has ICSP to conclude that $E/(H,S)$ has CSP. As the next paragraph shows, the statement of \cite[Theorem 4.3]{Ranga_prime} becomes correct if (3') is replaced by (3'') where (3'') denotes (3') with CSP replaced by ICSP. This would correct the argument on line 16 of page 82 of \cite{Ranga_prime} where CSP only, not ICSP, of $E^0-H$ is claimed to be sufficient for CSP of $(E/(H, B_H))^0.$ 
We also note that the statement of \cite[Theorem 4.3]{Ranga_prime} becomes correct if $E^0-H$ in condition (3') is replaced by $(E/(H,S))^0.$

If $P$ is non-graded, then $P$ is primitive if and only if (1') holds by \cite[Theorem 4.3]{Ranga_prime}.
If $P$ is graded and primitive, then $L_K(E)/P$ is primitive, so $(E/(H,S))^0$ is downwards directed, has CSP and Condition (L) by \cite[Theorem 4.7]{Gene_Bell_Ranga}. If $S=B_H,$ then $E^0-H$ is
downwards directed, has ICSP and Condition (L) by Lemma \ref{lemma_quotient_graph} and the proof of Corollary \ref{corollary_graded_primitive_and_primitive}, so (3'') holds. If $u\in B_H-S,$ then $B_H-S=\{u\}$ since  $(E/(H,S))^0$ is downwards directed. In addition, $P$ is graded and prime, so $P$ is graded prime and (2') holds. 
Conversely, if (2') holds, then the primness of $P$ implies that $E^0-H$ is downwards directed, so $E^0-H=R(u)$ holds 
and $E^0-H$ has ICSP.  Hence, both (2') and (3'') imply that $P$ is graded primitive by condition (2) of Theorem \ref{graded_primitive_ideal} and primitive by Corollary \ref{corollary_graded_primitive_and_primitive}.

\subsection{\texorpdfstring{\cite[{\bf Definition 3.1 and Example B}]{Roozbeh_Ranga}}{TEXT}}\label{subsection_RangaRoozbeh_branching}
In \cite[Definition 3.1]{Roozbeh_Ranga}, a branching system was defined as a saturated and perfect branching system in the sense used in  \cite{Chen} and in this paper. In \cite[Example B of \S 3]{Roozbeh_Ranga}, the module $\N_v$ for an infinite emitter $v$ was introduced using the same branching system we considered in section \ref{subsection_list_of_Chen_modules} (recall that $X$ consists of paths which end in $v$). However, since $v$ is not equal to $p$ for any path $p$ of positive length which starts and ends at $v$, $v\in X_v$ and $v\notin X_e$ for any $e\in \so^{-1}(v)$. Thus, $\bigcup_{e\in \so^{-1}(v)}X_e\subsetneq X_v$ showing that $X$ is not perfect.

\subsection{\texorpdfstring{\cite[{\bf Definitions 3.2 and 3.3}]{Roozbeh_Ranga}}{TEXT}}\label{subsection_RangaRoozbeh_N_vc}
Let $c$ be a cycle without exits and $v$ be a vertex on $c.$ In \cite[Definition 3.2]{Roozbeh_Ranga} the authors of \cite{Roozbeh_Ranga} introduce a branching system by considering the set $Z$ defined by \[Z=\{pq^*\mid p,q\mbox{  are paths and }\so(q)=v\},\] with $Z_w=\{pq^*\in Z\mid \so(p)=w\}$ for $w\in E^0,$ $Z_e=\{pq^*\in Z\mid e$ is the first edge of $p\}$ and $\sigma_e: Z_{\ra(e)}\to Z_e$ given by $pq^*\mapsto epq^*$ for $e\in E^1.$ 
Without specifying that the elements $pq^*$ are considered as elements of $L_K(E),$ this is not a branching system since the axiom (4) fails. Indeed, $v$ is an element of $Z_v$ which is not in $Z_e$ for any $e\in \so^{-1}(v).$ 
This would create problems with the $L_K(E)$-module structure as the following example shows. 
\begin{example}
Let $E$ be the graph 
$\xymatrix{\bullet^v\ar@(ur,dr)^e}.$ Let $e^0$ denote $v$ and $e^{-n}$ denote $e^*$ listed $n$ times for a positive integer $n.$ By definition of $Z,$ $Z$ consists of the elements of the form $e^me^{-n}$ where $m$ and $n$ are nonnegative integers, $Z=Z_v,$ $Z_e=\{e^me^{-n}\in Z\mid m>0\},$ and $\sigma_e: Z_v\to Z_e$ is the map $e^me^{-n}\mapsto e^{m+1}e^{-n}.$ The first three axioms of the branching system hold, but we have that $v\in Z_v$ and $v\notin Z_e,$ so axiom (4) fails. If $M(Z)$ is the $K$-vector space with basis $Z,$ let us consider the action of $e^*$ introduced as \cite[Definition 3.3]{Roozbeh_Ranga} by  
$e^*\cdot x=\sigma_e^{-1}(x)$ if $x\in Z_e$ and 0 otherwise. Then we have that $e^*\cdot v=0$ since $v\notin Z_e.$ However, as $ee^*=v$ in the algebra $L_K(E),$ we have that 
\[v=v\cdot v=ee^*\cdot v=e\cdot(e^*\cdot v)=e\cdot 0=0\]
which is a contradiction with the assumption that $v$ is a basis element. 
\label{example_not_branching} 
\end{example}

Changing the definition of $Z$ to   
\[Z=\{pq^*\in L_K(E)\mid p,q \mbox{ are paths and }\so(q)=v\},\]
also create problems since 0 is in $Z$ if $E$ contains paths $p$ and $q$ such that $\ra(p)\neq \ra(q)$ (for example, in the case when $E$ has more than one vertex). To avoid having zero in a basis set, one can require that $\ra(p)=\so(q)$ in the definition of $Z.$ Ultimately, if $Z$ is defined by 
\[Z=\{pq^*\in L_K(E)\mid p,q \mbox{ are paths with }\ra(p)=\ra(q)\mbox{ and }\so(q)=v\},\]
then $Z$ is the same as the set $X$ we introduced in Definition \ref{definition_branching_for_N_c}. As a consequence, if the original definition of $Z$ is changed to this last version,  then the module $\N_{vc}=M(Z)$ from \cite{Roozbeh_Ranga} is the same as $\N_c^v=M(X).$

\subsection{\texorpdfstring{\cite[{\bf Theorem 3.4}]{Roozbeh_Ranga}}{TEXT}}\label{subsection_RangaRoozbeh_graded_primitive}
\cite[Theorem 3.4]{Roozbeh_Ranga} states that an ideal $P$ with $H=E^0\cap P$ is graded primitive if and only if either (i) $P$ is a graded prime ideal of the form $I(H, B_H)$ such that $E/(H, B_H)$ satisfies Condition (L) and has CSP or (ii) $P$ is prime and $P=I(H, B_H-\{u\})$ for some $u\in B_H$ such that $E^0-H=R(u).$ This may seem to contradict Theorem \ref{graded_primitive_ideal}. However, ``graded primitive'' in \cite{Roozbeh_Ranga} is used in the sense ``graded and primitive'', not in the sense used in this paper. With this in mind, \cite[Theorem 3.4]{Roozbeh_Ranga} and Theorem \ref{graded_primitive_ideal} do not contradict each other. In fact, \cite[Theorem 3.4]{Roozbeh_Ranga} also follows from Corollary \ref{corollary_primitive_ideal}. 

\subsection{\texorpdfstring{\cite[{\bf Theorem 3.5}]{Roozbeh_Ranga}}{TEXT}}\label{subsection_RangaRoozbeh_annihilator}
If $c$ is a cycle with no exits and $H=E^0-R(c^0),$ part (2) of \cite[Theorem 3.5]{Roozbeh_Ranga} states that  $\ann(\N_{vc})$ is $I(H, \emptyset).$ If we assume that $\N_{vc}$ is defined so that none of the issues of section \ref{subsection_RangaRoozbeh_N_vc} are present, then it matches the module $\N_c^v.$ The following example shows that the set $S$ from $\ann(\N_c^v)=I(H, S)$ can be nonempty. In particular, $S=B_H$ by Proposition \ref{proposition_N_c}.

\begin{example}
Let $E$ be the graph 
$\;\xymatrix{\bullet^v \ar@(ul,dl)_c  & \bullet^u\ar[l]_e
\ar@{.} @/_1pc/ [r] _{\mbox{ } } \ar@/_/ [r] \ar [r] \ar@/^/ [r] \ar@/^1pc/ [r] &\bullet^w  }.$ Then $R(v)=\{u, v\}$ and for $H=E^0-R(v)=\{w\},$ $B_H=\{u\}.$ The annihilator $I(H,S)$ of $\N^v_c$ contains $u^H=u-ee^*.$ Indeed, since 
an element of $X_u$ has the form $ec^n$ for $n\in \Zset$ where $c^{-1}$ stands for $c^*$ and $c^0$ stands for $v,$ we have that
$(u-ee^*)\cdot ec^n=ec^n-e\cdot c^n=ec^n-ec^n=0$ for every $n\in \Zset.$ So, $S=\{u\}\neq \emptyset.$ Note that $\ann(\N^v_c)=I(\{w\}, \{u\})$ also follows from Proposition \ref{proposition_N_c}. 
\label{example_RoozbehRanga_Thm_3_5}
\end{example}

The problem in the proof of \cite[Theorem 3.5]{Roozbeh_Ranga} is on line 14 of page 473: the sum in the formula on this line should be taken over $e\in \so^{-1}(v), \ra(e)\notin H$, not over  $e\in \so^{-1}(v), \ra(e)\in H,$ by the definition of the set $B_H.$ With the correct subscript, the relation $u^Hp=0$ holds instead of  $u^Hp=p.$

\end{document}